\documentclass[13pt]{amsart}
 \usepackage{amsmath}
 \usepackage{}
 \usepackage{amsfonts}
 \usepackage{amsfonts,latexsym,rawfonts,amsmath,amssymb,amsthm}
 \usepackage[plainpages=false]{hyperref}

 \usepackage{pdfsync}

 \usepackage[a4paper]{geometry}
 \usepackage{esint}
 \usepackage{graphics,graphicx}
 \usepackage{tikz}

 \numberwithin{equation}{section}

 \newcommand{\beq}{\begin{equation}}
 \newcommand{\eeq}{\end{equation}}
 \newcommand{\beqs}{\begin{eqnarray*}}
 \newcommand{\eeqs}{\end{eqnarray*}}
 \newcommand{\beqn}{\begin{eqnarray}}
 \newcommand{\eeqn}{\end{eqnarray}}
 \newcommand{\beqa}{\begin{array}}
 \newcommand{\eeqa}{\end{array}}

 \def\lra{\longrightarrow}

 \def\bc{\begin{center}}
 \def\ec{\end{center}}

 \def\begeq{\begin{equation}}
 \def\endeq{\end{equation}}
 \def\and{\quad{\rm and}\quad}

 \let\lra=\longrightarrow

 \def\mapright\#1{\,\smash{\mathop{\lra}\limits^{\#1}}\,}

 \newtheorem{prop}{Proposition}[section]
 \newtheorem{theo}[prop]{Theorem}
 \newtheorem{lem}[prop]{Lemma}
 
 \newtheorem{cor}[prop]{Corollary}
 \newtheorem{rem}[prop]{Remark}
 
 \newtheorem{defi}[prop]{Definition}
 \newtheorem{conj}[prop]{Conjecture}

\begin{document}

  \title{ Tian's partial $C^0$-estimate implies Hamilton-Tian's conjecture}


 \author{Feng   Wang  and Xiaohua $\text{Zhu}^{*}$}


 \address{Department of Mathematics,  Zhejiang
 University, Hangzhou 310028, China.}

 \email{ wfmath@zju.edu.cn}

 \address{School of Mathematical Sciences, Peking
 University, Beijing 100871, China.}

 \email{ xhzhu@math.pku.edu.cn}

 \thanks {* Partially supported by NSFC Grants 11771019 and BJSF Grants Z180004.}

 \begin{abstract}
 In this paper, we prove the    Hamilton-Tian  conjecture for  K\"ahler-Ricci flow  based on a recent work   of Liu-Sz\'ekelyhidi on   Tian's  partical $C^0$-estimate for  poralized K\"ahler metrics with Ricci bounded  below. The Yau-Tian-Donaldson conjecture for the existence of K\"ahler-Einstein metrics  on Fano manifolds  will be also discussed.
 \end{abstract}

 \maketitle

  \date{}

  \section{Introduction}
 Let $(M, J, \omega_0)$ be a Fano manifold and $\omega_t$  a solution of normalized K\"ahler-Ricci flow,
 \begin{align}\label{kr-flow}
 \frac{\partial \omega_t}{\partial t}=-{\rm Ric}\,(\omega_t)+\omega_t,~\omega_0\in 2\pi c_1(M, J).
 \end{align}
 The Hamilton-Tian  conjecture asserts \cite{Tian97}:

  {\it Any sequence of $(M, \omega_t)$ contains a subsequence converging to a length space $(M_\infty,\omega_\infty)$
 in the Gromov-Hausdorff topology and $(M_\infty,\omega_\infty)$ is a smooth K\"ahler-Ricci soliton outside a closed subset $S$, called the singular set, of codimension at least $4$. Moreover, this subsequence of $(M, \omega_t)$ converges  locally to the regular part of $(M_\infty,\omega_\infty)$
 in the Cheeger-Gromov topology.}

  Recall that a K\"ahler-Ricci soliton on a complex manifold $M$ is a pair $( \omega, v)$, where $v$ is a holomorphic vector field on $M$ and $\omega$ is a K\"ahler metric on $M$, such that
 \begin{align}\label{kr-soliton}{\rm Ric}(\omega)\,-\,\omega\,=\,L_v \omega,
 \end{align}
 where $L_v$ is the Lie derivative along $v$. If $v=0$, the K\"ahler-Ricci soliton becomes a
 K\"ahler-Einstein metric. The uniqueness theorem in \cite{TZ1} states that a K\"ahler-Ricci
 soliton on a compact complex manifold, if it exists, must be unique modulo    auto-morphisms  group ${\rm Aut}(M)$.\footnote {In the case of K\"ahler-Einstein metrics, this uniqueness theorem is due to Bando-Mabuchi \cite{BM}.} Furthermore, $v$ lies in the center of Lie algebra of a reductive part of  ${\rm Aut}(M)$.

 The Gromov-Hausdorff convergence part in the Hamilton-Tian conjecture follows from Perelman's non-collapsing result and Zhang's upper volume estimate \cite{Zhq}. More recently, there were very significant progresses on this conjecture, first by Tian and Zhang in dimension less than $4$ \cite{TZhzh}, then by Chen-Wang \cite{Chwang} and Bamler \cite{Bam} in higher dimensions. In fact, Bamler proved a generalized version of the conjecture.

 The purpose of  paper is to give a   new  proof to  the Hamilton-Tian conjecture by using a recent result  of Liu-Sz\'ekelyhidi on  Tian's  partical $C^0$-estimate for  poralized K\"ahler metrics with Ricci bounded  below \cite{LS}.
 More precisely, we prove

   \begin{theo}\label{WZ} For any sequence of $(M, \omega_t)$ of  (\ref{kr-flow}),
  there is a subsequence $t_i\rightarrow \infty$  and a Q-Fano variety $\tilde M_\infty$ with klt singularities such that  $\omega_{t_i}$ is locally  $C^\infty$-convergent to a  K\"ahler-Ricci soliton  $\omega_\infty$ on  ${\rm Reg}(\tilde M_\infty)$ in the Cheeger-Gromov topology.  Moreover,    $\omega_\infty$ can be extended  to a  singular  K\"ahler-Ricci soliton  on $\tilde M_\infty$ with  a $L^\infty$-bounded K\"ahler potential $\psi_\infty$ and the completion of $({\rm Reg}(\tilde M_\infty), \omega_\infty)$ is isometric to the global limit  $(M_\infty, \omega_\infty')$ of  $\omega_{t_i}$  in the Gromov-Hausdorff topology.   In  addition,    if $\omega_\infty$ is a  singular  K\"ahler-Einstein  metrics,  $\psi_\infty$ is continuous and $M_\infty$   is homeomorphic to  $\tilde M_\infty$ which has  at least 4   Hausdroff codimension  of singularities  of  $(M_\infty, \omega_\infty')$.
  \end{theo}

 Compared to the proofs in the two long  papers  \cite{Chwang} or \cite{Bam},
   our proof of  Theorem \ref{WZ} is  purely  analytic  by using the technique of  complex
 Monge-Amp\`ere equation.  In particular,  we  avoid to use the hard deeply Perelman's pseudo-locality theorem, which plays a critical role
  in  their  papers.
  The pseudo-locality theorem was established for the blow-up argument to the Hamilton's Ricci flow   by Perelman in his celebrated  paper \cite{Pe}. In   Theorem \ref{WZ}, we also obtain    a structure of   Q-Fano variety  with klt singularities to the Gromov-Hausdorff limit in  the Hamilton-Tian  conjecture.

 There are many  works on Tian's partial $C^0$-estimate \cite{Ti90, Ti12, DS, T2, Ti13, Ji13,  JWZ, Sz16, DS16}, etc.   This estimate
 has been a   powerful tool   in   solving   problems in K\"ahler geometry.  For examples,  Tian proved the Yau-Tian-Donaldson conjecture  for the existence of K\"ahler-Einstein metrics  on Fano manifolds by establishing  the partial $C^0$-estimate for conical  K\"ahler-Einstein metrics \cite{T2}.  Chen, Donaldson and Sun  gave an alternative proof  of  the conjecture,  independently  \cite{CDS}. The  conjecture  asserts
 that a Fano manifold  admits a K\"ahler-Einstein metric  if and only if it is $K$-polystable.
 Our proof of Theorem \ref{WZ} also shows that the partial $C^0$-estimate implies   Hamilton-Tian's conjecture.

 It is interesting in giving some remarks on the limit  $M_\infty$ in  Theorem \ref{WZ} (also see Proposition  \ref{two-equiv}).   In general, $M_\infty$ could not be smooth as   shown recently for  some examples  from K\"ahler-Ricci flow on Fano  compactifications of  Lie groups  \cite{LTZ}.  Moreover,  even for the smooth case of $M_\infty$,  the complex structure of  underlying  manifold may  jump up (cf. \cite{TZ4, SW, Wang}). However, to the  best  of   our   knowledge, there are very few references about the uniqueness of $M_\infty$ except   that $M_\infty$ is a  (smooth) K\"ahler-Einstein manifold  \cite{TZ1, TZ4, TZZZ}.  For  examples,  in  a special case that $M_\infty$ is a  smooth  manifold with a K\"ahler-Ricci soliton (but not a K\"ahler-Einstein metric) and a  different complex structure to $M$,  it is still unknown whether  $M_\infty$ depends on a sequence of evolved metrics or an initial metric of K\"ahler-Ricci flow.

 As an application of Theorem \ref{WZ}, we give an approach of Yau-Tian-Donaldson conjecture by using   K\"ahler-Ricci flow.

 \begin{cor}\label{WZ-2}  Let  $M$ be a $K$-polystable  Fano manifold.  Suppose that  the Mabuchi $K$-energy is bounded below.
  Then   $(M, \omega_t)$ in (\ref{kr-flow})  converges to a  K\"ahler-Einstein metric in $C^\infty$-topology after holomorphisms transformation.   In particular,   $M$ admits a K\"ahler-Einstein metric.
 \end{cor}

 By work of  Li-Sun \cite{LS},  the $K$-semistability   is equivalent to the lower bound of $K$-energy.  However, their proof  shall depend on the proof of  Yau-Tian-Donaldson's conjecture in  \cite{T2} or  \cite{CDS}.  Thus if there is a proof of K\"ahler-Ricci flow for Li-Sun's result, Corollary \ref{WZ-2}   will give  a proof  of Yau-Tian-Donaldson conjecture by  K\"ahler-Ricci flow.

 The paper is organized as follows.   In Section 2, we  first recall    Liu-Sz\'ekelyhidi's  results on  Tian's  partical $C^0$-estimate (cf. Theorem \ref{LS} and Theorem \ref{normal}), then we apply them to the  evolved metrics $\omega_t$  of (\ref{kr-flow})   to get a  metrics  comparison to  the  induced metrics by the Fubini-Study  metric (cf. Proposition \ref{q-fano-limit} and
 Proposition \ref{metric-equivalent}).  Section 3 is devoted to get a local  higher regularity  for K\"ahler potentials  (cf. Proposition \ref{regularity-phi}). Theorem \ref{WZ} and Corollary \ref{WZ-2}  will be proved in Section 4.

 \vskip3mm

 \noindent {\bf Acknowledgements.}We would like to thank professor Gang Tian for inspiring conversations.

 \section{Partial $C^0$-estimate}

 \subsection{Liu-Sz\'ekelyhidi's work}
 Let $(M, L, \omega)$ be a polarized manifold such that $\omega$ is a K\"ahler metric in $2\pi c_1(L)$. Choose a hermitian metric $h$ on $L$ such that ${\rm R }(h)=\omega.$ Then for any positive integer $l$, we have the $L^2$-metric on $H^0(M,L^{l}, \omega)$:
 \begin{align}\label{inner-product}
 ( s_1,s_2)=\int_M\langle s_1,s_2\rangle_{h^{\otimes l}}\omega^n,
 \end{align}
 for $s_1, s_2\in H^0(M,L^{\otimes l} )$.

 \begin{defi}
 We denote this linear space $H^0(M,L^{l},\omega)$  with inner product by  (\ref{inner-product}).
 For any orthonormal basis $\{s^\alpha\} ~({0\leq \alpha\leq N=N(l)})$ of $H^0(M,L^{l},\omega)$, we define the Bergman kernel as
 \begin{align}
 \rho_l(M,\omega)(x)\,=\,\Sigma_{i=0}^N \,|s^\alpha|^2_{h^{\otimes l}}(x), ~\forall~ x\in M.
 \end{align}
 \end{defi}

 Recently,   Liu-Sz\'ekelyhidi proved the following in \cite{LS},

 \begin{theo}\label{LS}
 Given $n,D,v>0$, there is a positive integer $l$ and a real number $b>0$ with the following property: Suppose that $(M,L,\omega)$ is a polarized K\"ahler manifold with $\omega\in 2\pi c_1(L)$ such that
  $$ {\rm Ric}\,(\omega)\geq  \omega, ~{\rm vol}(M,\omega )\geq v, ~{\rm diam}(M,\omega)\leq D.$$
  Then for any $x\in M$ it holds,
  \begin{align}\label{partial-c0}\rho_l(M,\omega)(x)\geq b.
  \end{align}
 \end{theo}

 Such an inequality  (\ref{partial-c0}) was called the partial $C^0$-estimate by Tian \cite{Ti90, Ti90a, Ti90b, Ti12}.  The upper bound of
 $\rho_l(M,\omega)$ can be also obtained by using the standard Moser iteration (for examples, see \cite[Lemma 3.2]{JWZ}).
  By (\ref{partial-c0}),  we can write $\omega$ as a metric with bounded  K\"ahler potential using the Fubini-Study metric as the background metric. In fact, if  the orthonormal basis $\{s^\alpha\}$ $(0\leq \alpha\leq N)$ defines an embedding $\Phi$,  then we have
   $$\omega=\Phi^*(\frac{1}{l}\omega_{FS})- \frac{1}{l}\sqrt{-1}
   \partial \bar\partial \log \rho_l(M,\omega).$$
    By the gradient estimate of $s^\alpha$   (cf. \cite{Ti90a, DS, Ti13}) and the lower bound  (\ref{partial-c0})  of $\rho_l(M,\omega)$,  it holds that
 \begin{align}\label{lower-bound-metric}\Phi^*(\frac{1}{l}\omega_{FS})\leq C(n,D,v) \omega.
 \end{align}
 In fact, we have
 $$\Phi^*(\omega_{FS})=\sqrt{-1}\frac{\sum_{\alpha=0}^N\langle \nabla s^\alpha,\nabla s^\alpha\rangle}{\rho_l(M,\omega)}-\sqrt{-1}\frac{(\sum_{\alpha=0}^N\langle \nabla s^\alpha, s^\alpha\rangle)
 (\sum_{\alpha=0}^N\langle s^\alpha,\nabla s^\alpha\rangle)}{\rho^2_l(M,\omega)}.$$

 Based on the partial $C^0$-estimate,     Liu-Sz\'ekelyhidi  also proved \cite{LS},

 \begin{theo}\label{normal}
 Given $n,d,v>0$, let $(M^n_i, L_i, \omega_i)$ be a sequence of polarized K\"ahler manifolds such that
 \begin{itemize}
     \item $L_i$ is a Hermitian holomorphic line bundle with hermitian metric $h_i$ such that $R(h_i)=\omega_i$;

     \item ${\rm Ric}(\omega_i) \geq -\omega_i$ and

   \begin{align}\label{vol-diamerter-condition} {\rm vol}(M_i, \omega_i)>v, ~{\rm diam}(M_i,\omega_i) <d;
    \end{align}
   \item The sequence $(M^n_i ,\omega_i)$ converges  to a limit metric space $(X, d_\infty)$  in the Gromov-Hausdorff topology.
 \end{itemize}
  Then $X$ is homeomorphic to a normal variety. More precisely, there is a positive integer $l=l(n,d,v)$ such that  the orthonormal basis of $H^0(M_i, L_i^{l}, \omega_i)$ defines an embedding $\bar\Phi_i$, and  $\bar\Phi_i(M_i)$  ( perhaps replaced by taking a subsequence) converges to a normal variety $W$ in $\mathbb CP^N$ as cycles, which   is homeomorphic to $X$.
 \end{theo}

 \begin{rem}\label{LS-theorem}According to the proof of  Theorem \ref{normal}  \cite{LS},  the limit  $\bar\Phi_\infty$ of maps $\bar\Phi_i$ actually gives a  homeomorphism from $X$ to $W$.
 \end{rem}

 We can choose the same positive integer in the above  two theorems.

 Applying  Liu-Sz\'ekelyhidi's result,  Theorem \ref{LS} to the K\"ahler metrics evolved in  (\ref{kr-flow}),   Zhang  proved  very recently  in \cite{ZhK},

 \begin{prop}\label{partial-zhang} Assume that  $(M,\omega_t)$ is the solution of (\ref{kr-flow}).
 Then there exists a positive integer $l=l(M,\omega_0)$ and a positive real number $b=b(M,\omega_0)$ such that $\rho_l(M,\omega_t)(x)\geq b$ for any $x\in M$.
 \end{prop}

 By  Proposition \ref{partial-zhang},  one can define  a holomorphic map $\Phi_t: M\to \mathbb CP^N$ by an orthonormal  basis of   $H^0(M,  K^{-l}_M, $
  $\omega_t)$
 for any $t\in(0, \infty)$.   In the following, we  apply  Theorem \ref{normal} to get a normal variety structure for the limit of images $\Phi_t(M)$  in $\mathbb CP^N$ in sense of algebraic geometry.

 \subsection{Partial $C^0$-estimate for K\"ahler-Ricci flow}

 Let $\omega_t=\omega_0+\sqrt{-1}\partial \bar\partial \phi_t$ be the solution of (\ref{kr-flow}). Then
 $${\rm Ric}\,(\omega_t)-\omega_t=\sqrt{-1}\partial \bar\partial (-\dot \phi_t ).$$
  The following estimates are due to G. Perelman. We refer the  reader to
 \cite{ST} for their proof.

 \begin{lem}\label{lem:perelman-1}  There are constants $c>0$ and $C>0$ depending only on the initial metric $\omega_0$ such that the following is true:

 1) ${\rm diam}(M,\omega_t)\le C$, ~
 ${\rm vol}(B_r(p),\omega_t )\ge c r^{2n}$;

 2) For any $t\in (0,\infty)$, there is a constant $c_t$ such that  $h_t=-\dot \phi_t +c_t$ satisfies
 \begin{align}\label{h-t-estimate}  \|h{_t}\|_{C^0(M)}\le C,
 ~\|\nabla h_{t}\|_{\omega_{t}} \le C,
 ~ \|\Delta h_{t}
 \|_{C^0(M)} \le C.
 \end{align}
  \end{lem}

 \begin{prop}\label{q-fano-limit}
 Given any sequence $t_i\rightarrow \infty$, there is a subsequence, which is still denoted as $t_i$, such that the embedding of $M$   by an orthonormal  basis  $\{s^\alpha_{t_{i}}\}$ of   $H^0(M,  K^{-l}_M, \omega_{t_i})$
  in $\mathbb CP^N$ converges to a  normal variety $\tilde M_\infty$.
 \end{prop}

 \begin{proof}As in \cite{ZhK}, we
 let $\eta_t$ be a solution of
 \begin{align}\label{modified-metric}{\rm Ric}\,(\eta_t)=\omega_t.
 \end{align}
 Writing $\eta_t= \omega_t+\sqrt{-1}\partial \bar \partial \kappa_t$,  we have
 \begin{align}\label{yau-equation}
 (\omega_t+\sqrt{-1}\partial \bar \partial \kappa_t)^n=e^{h_t}\omega_t^n,~\sup_M\kappa_t=0,
 \end{align}
 where $h_t$ is the  Ricci potential of $\omega_t$ chosen as in Lemma \ref{lem:perelman-1}.
 By the Yau's solution to Calabi's  problem \cite{Yau},   (\ref{yau-equation})  can be solved. Moreover,
 since
 \begin{align}\label{h-t-estimate-0}  \|h{_t}\|_{C^0(M)}\le C,
 \end{align}
 by the moser iteration  (cf. \cite{TZ1}) together with  Zhang's  Sobolev inequality \cite{Zh},
 \begin{align}\label{c0-psi-kappa}
 |\kappa_t|\le C( \|h{_t}\|_{C^0(M)}, \omega_0)\le A.
 \end{align}

  By Lemma \ref{lem:perelman-1}-2),  it has been verified by Zhang  in \cite{ZhK}
  that  (\ref{vol-diamerter-condition}) is satisfied for metrics $\eta_t$.
 Then applying  Theorem \ref{normal},   there is a positive integer $l$ such that each  orthonormal basis $\{\bar s_{t_i}^\alpha\}$ of $H^0(M, K_M^{-l}, \eta_{t_i})$ defines an embedding $\bar \Phi_i:M\to \mathbb CP^N$.   By choosing  a subsequence
 $t_i\rightarrow \infty$ if necessary,   $\bar \Phi_i(M)$ converges to $W$ which is a normal variety.
  On the other hand, by   (\ref{h-t-estimate-0}) and  (\ref{c0-psi-kappa}),  if we let
  $\{s_{t_i}^\alpha\}$ $(0\leq \alpha\leq N(l))$ be the orthonormal basis of $H^0(M, K_M^{-l}, \omega_{t_i})$, there is an invertible matrix $A(t_i)=[a_\alpha^\beta(t)]$ such that $ s^\beta_{t_i}=a_\alpha^\beta(t_i) \bar s_{t_i}^\alpha$ and
  \begin{align}\label{matrix}
  \|A(t_i)\|\leq C, \|A^{-1}(t_i)\|\leq C,
  \end{align}
  where  $C>0$ is a uniform constant.    Thus by  taking a subsequence, we have $A(t_i)\rightarrow A\in SL(N+1;\mathbb C)$. Hence,  the embedding by $\{s^\alpha_{t_i}\}$ converges to $A\circ W=\tilde M_\infty$.  The proposition is proved.
 \end{proof}

 We note  by (\ref{h-t-estimate-0})   and    (\ref{matrix}) that  there is  also a  partial $C^0$-estimate for the metrics $\omega_{t_i}$ as in  Proposition  \ref{partial-zhang}. Namely,
 \begin{align}\label{partial-WZ}
 \rho_l(M,\omega_i)(x)\, &=\,\Sigma_{i=0}^N \,|s_{t_i}^\alpha|^2_{h^{\otimes l}}(x)\ge c_0 \,\Sigma_{i=0}^N \,|\bar s_{t_i}^\alpha|^2_{h^{\otimes l}}(x)\notag\\
 &= c_0 \rho_i(M,\eta_i)(x)\ge c_0' >0.
 \end{align}
 Denoting the embedding of $\{s^\alpha_{t_{i}}\}$ by $\Phi_i$, we have $\tilde M_i=\Phi_i(M)$ converging  to a normal variety $\tilde M_\infty$  by Proposition \ref{q-fano-limit}.  Thus
    \begin{align}\label{C^0}
   \omega_{t_i}=( \Phi_i)^*(\frac{1}{l}\omega_{FS})-\sqrt{-1}\partial\bar\partial \phi_i,   \end{align}
 where $\phi_i= \frac{1}{l} \log\rho_l(M,\omega_i)$ and it satisfies
 \begin{align}\label{c0-phi}|\log \phi_i|\leq C.
 \end{align}
 Moreover, by  the   gradient  estimate for $s^\alpha_{t_{i}}$  \cite[Lemma 3.1]{JWZ},
 \begin{align}\label{section-estiamte-1}
 \|\nabla s^\alpha_{t_{i}}\|_{\omega_{t_i}}\leq C(   \|h{_{t_i}}\|_{C^0(M)}, C_s,  n)l^{\frac{n}{2} +1},
 \end{align}
 where $C_s$ is the Sobolev constant of $(M,\omega_{t_i})$,  which is uniformly controlled \cite {Zh}.
 Thus as in (\ref{lower-bound-metric}), we have
   \begin{align}\label{upper-bound} \frac{1}{l} \omega_{FS}|_{\tilde M_i}\le C     (\Phi_i^{-1})^*\omega_{t_i}.
    \end{align}

 Next we want to extend the relations (\ref{c0-phi}) and (\ref{upper-bound}) for any metrics $\omega_{t_i+s}$, where $s\in [-1, 1]$.
 Let  $\{s^\alpha_{t_{i}+s}\}$  be   an  orthonormal basis    of $H^0(M, K_M^{-l}, \omega_{t_i+s})$ which gives a  Kodaira embedding  $\Phi_i^s: M\to \mathbb CP^N$.  We show

  \begin{prop}\label{metric-equivalent}For any  $s\in [-1, 1]$, it holds
   \begin{align}\label{upper-metric-2}  \frac{1}{l}\omega_{FS}|_{\tilde M_i}\le C( ( \Phi_{i}^s)^{-1})^*\omega_{t_i+s}.
   \end{align}
   Moreover, we can write
    \begin{align}\label{equivalent}  (  \Phi_{i}^{-1})^*\omega_{t_i+s}=\frac{1}{l}\omega_{FS}|_{\tilde M_i} +\sqrt{-1}\partial\bar\partial \psi_i^s,
   \end{align}
   where  $\psi_i^s$ is uniformly bounded  independently  of $i$ and $s$.
 \end{prop}

 \begin{proof} First we assume $s\in [0, 1]$.
   Denoting $\omega_{t_i+s}=\omega_{t_i}+\sqrt{-1}\partial \bar\partial \phi_s$, we have
 \begin{align}\label{equation-s}
 \frac{\partial \phi_s}{\partial s}=\log \frac{(\omega_{t_i}+\sqrt{-1}\partial \bar\partial \phi_s)^n}{\omega^n_{t_i}}+\phi_s-h_{t_i},~\phi_{s=0}=0.
 \end{align}
 Taking the derivative at $s$, we get
 $$\frac{d}{ds}\dot \phi_s=\Delta_{\omega_{t_i+s}} \dot \phi_s+\dot \phi_s,~s\in [0, 1].$$
 Thus by the maximum principle,
 \begin{align}\label{c0-phis}| \phi_s|\le C,  |\dot \phi_s|\le C.
 \end{align}
 As a consequence,
 \begin{align}\label{volume-s}
 |\log \frac{(\omega_{t_i}+\sqrt{-1}\partial \bar\partial \phi_s)^n}{\omega^n_{t_i}}|\le C, ~\forall ~s\in [0, 1],
 \end{align}
   which means the volume elements are equivalent between $\omega_{t_i+s}$ and $\omega_{t_i}$. Hence,  as in the proof of
   Proposition \ref{q-fano-limit},  there is an invertible matrix $A(s)=[a_\alpha^\beta(t_i+s)]$ such that $\ s^\beta_{t_i+s}=a_\alpha^\beta(t_i+s) s_{t_i}^\alpha$ and
  \begin{align}\label{matrix-2}
  \|A(s)\|\leq C, \|A^{-1}(s)\|\leq C.
  \end{align}
  where  $C>0$ is a uniform constant independent of $i, s$.
   Therefore,  there is a uniform constant $c>0$ such that
   \begin{align}\label{fd-mretric}   c^{-1} \omega_{FS}|_{\tilde M_i}\le  (\Phi_i^s \cdot\Phi_i^{-1} )^*\omega_{FS} \le c \omega_{FS}|_{\tilde M_i}.
   \end{align}

  By (\ref{c0-phis}) and (\ref{volume-s}),    similar to (\ref{c0-phi}),    we also have the partial $C^0$-estimate
  for metrics $\omega_{t_i+s}$ and
   \begin{align}\label{partial-WZ-3}
  | \log \rho_l(M,\omega_{t_i+s})|\le A_0.
  \end{align}
 Note that as in   (\ref{section-estiamte-1})
     each section  $s^\alpha_{t_{i}+s}$  satisfies
  \begin{align}\label{section-estiamte-2}
 \|\nabla s^\alpha_{t_{i}+s}\|_{h}\leq C(   \|h{_{t_i+s}}\|_{C^0(M)}, C_s,  n)l^{\frac{n}{2} +1}.
 \end{align}
   Thus  we have an analogy of  (\ref{upper-bound}),
 \begin{align}\label{upper-bound-2} \frac{1}{l} \omega_{FS}|_{\Phi_i^s(M)}\le C     ((\Phi_i^s)^{-1})^*\omega_{t_i+s}.
    \end{align}
 Combining this  with (\ref{fd-mretric}), we derive (\ref{upper-metric-2}).

 It is easy to see that
    \begin{align}\label{c0-psi}\psi_i^s= \log \frac{\Sigma_{i=0}^N |s_{t_i+s}^\alpha|^2}{\Sigma_{i=0}^N \,|\bar s_{t_i}^\alpha|^2}
    -\log \rho_l(M,\omega_{t_i+s}) (\Phi_i^{-1}(\cdot)).
    \end{align}
    Then  by  (\ref{matrix-2}) and (\ref{partial-WZ-3}) together with   (\ref{matrix}),  we  get
    \begin{align}\label{phi-s-c0}
   |  \psi_i^s|\le C.
   \end{align}

 Next we consider $ s\in [-1, 0]$. Then we rewrite  (\ref{equation-s}) as
 \begin{align}\label{equation-s-1}
 \frac{\partial \phi_{s'}}{\partial s'}=\log \frac{(\omega_{t_i-1}+\sqrt{-1}\partial \bar\partial \phi_{s'})^n}{\omega^n_{t_i-1}}+\phi_{s'}-h_{t_i-1},~\phi_{s'=0}=0,~s'\in [0, 1].
 \end{align}
 Thus (\ref{c0-phis}) still holds.  Henc we can  follows the argument in  the case  for $ s\in [0, 1]$ to  get (\ref{phi-s-c0}).

 \end{proof}

 \section{Local higher $C^{k,\alpha}$-estimate}

 In this section, we derive locally  higher $C^{k,\alpha}$-estimate for $ \psi_i^s$ in Proposition \ref{metric-equivalent}.
   Let's  choose  an exhausting open sets $\Omega_\gamma\subset  \tilde M_\infty$.  Then by Proposition \ref{q-fano-limit},  there are  diffeomorphisms $\Psi_\gamma^i: \Omega_\gamma\to \tilde M_i$ such that
  the curvature of $ \omega_{FS}|_{\tilde {\Omega}_\gamma^i}$ is $C^k$-uniformly bounded independently of $i$ , where  $\tilde {\Omega}_\gamma^i=\Psi_\gamma^i(\Omega_\gamma)$.

 For simplicity,  we let $\tilde\omega_i=\frac{1}{l}\omega_{FS}|_{\tilde M_i}$.  Then by (\ref{equivalent}), we have
  $$(\Phi_i^{-1})^*\omega_{t_i+s}=\tilde\omega_{i}+\sqrt{-1}\partial\bar\partial \psi^s_i,   ~{\rm in} ~\tilde M_i, \forall s\in [-1,1].$$

  \begin{lem}\label{c3-phi}
  There exist  constants $A, C_\gamma, A_\gamma$ such that for $s\in [-1,1]$,
  \begin{align}
  &| \psi^s_i|\le A, ~{\rm in} ~\tilde M_i, \label{c0-estimate}\\
  &C_\gamma^{-1}\tilde\omega_{i}\le(\Phi_i^{-1})^*\omega_{t_i+s}\le C_\gamma \tilde\omega_{i}, ~{\rm in}~\tilde \Omega_\gamma^i,
  \label{c2-estimate} \\
  &\| \psi^s_i\|_{C^{3,\alpha}(\tilde {\Omega}_\gamma^i)} \le A_\gamma.\label{c3-estimate}
  \end{align}
 \end{lem}

 \begin{proof}
 Let $\tilde h_i$ be a Ricci potential of $\tilde\omega_{i}$. Then
  $\psi_i^s$ satisfies the complex Monge-Amp\`ere equation,
 \begin{align}\label{potenial-equation}
 (\tilde\omega_{i}+\sqrt{-1}\partial\bar\partial \psi^s_i)^n=
 e^{-\psi_i^s+\tilde h_i-h_i^s}\tilde\omega_{i}^n, ~{\rm in}~ \tilde M_i,
 \end{align}
 where   $h_i^s=h_{t_i+s}\circ \Phi_i^{-1}$ is the Ricci potential of $(\Phi_i^{-1})^*\omega_{t_i+s}$, which is uniformly bounded by Lemma \ref{lem:perelman-1}. Thus by (\ref{upper-bound}), we get (\ref{c2-estimate}).
  On the other hand, by
  $$\Delta{\omega_{t_i+s}}h_i^s=R_{t_i+s}-n,$$
  we have
  $$|\Delta_{\tilde\omega_{i}}h_i^s|\leq C,~ ~{\rm in}~\tilde \Omega_\gamma^i.$$
  It follows that
  \begin{align}\label{c1-hi}
 \| h^s_i\|_{C^{1,\alpha}(\Omega')} \le C_\gamma( d(\Omega')), ~\forall ~\Omega'\subset\subset\tilde\Omega_\gamma^i,
 \end{align}
 where  $d(\Omega')={\rm dist}(\Omega', \tilde\Omega_\gamma^i).$
 Hence, the regularity of (\ref{potenial-equation})  implies that
 $$\|\psi^s_i \|_{C^{3,\beta}(\Omega') }\le C_\gamma (d(\Omega')).$$
 Since we may replace $\tilde\Omega_\gamma^i$ by some bigger set $\tilde \Omega_{\gamma'}^i$, we get (\ref{c3-estimate}).
 \end{proof}

 We need to improve (\ref{c3-estimate}) in Lemma \ref{c3-phi} to

 \begin{prop}\label{regularity-phi}
 For any $k$ and $\Omega'\subset\tilde {\Omega}_\gamma^i$, it holds
 $$\| \psi_i^s\|_{C^{k,\alpha}(\Omega')} \le A_\gamma( d(\Omega'), k), ~\forall ~s\in [-\frac{1}{2}, 1].$$
 \end{prop}

 \begin{proof} As above,  we denote the Ricci potential of $\tilde \omega_i$ by $\tilde h_i$.  Then  it is easy to see that
  $\psi_i^s$ satisfies
 \begin{align}\label{s-equation}
 \dot \psi_i^s=\log \frac{(\omega_{i}+\sqrt{-1}\partial\bar\partial \psi_i^s)^n}{ \tilde\omega_{i}^n } +\psi_i^s
 -\tilde h_i, ~\forall ~(x,s)\in \tilde M_i\times [-1,1].
 \end{align}
 It follows that
 \begin{align}\label{equation-flow}\frac{d}{ds}\dot \psi_i^s=\Delta_{\omega_{i}^s}\dot \psi_i^s+\dot \psi_i^s.
 \end{align}

 Now we consider the restricted  flow  of (\ref{s-equation}),
 \begin{align}\label{s-equation-2} &\dot \psi_i^s=\log(\frac{(\tilde\omega_{i}+\sqrt{-1}\partial\bar\partial \psi^s_i)^n}{\tilde\omega_{i}^n})
 +\psi_i^s+\tilde h_i, ~\forall ~(x,s)\in \tilde \Omega^i_\gamma\times [{-1},1],
 \end{align}
 where $\tilde\omega_{i}$ is uniformly  $C^k$-bounded metric  on $\tilde \Omega_\gamma^i$ and $\psi_i^s$ satisfies (\ref{c3-estimate}).
 We claim that for $\Omega'\subset \subset \Omega_\gamma^i$:
 \begin{align}\label{c3-dot}\|\dot \psi_i^s\|_{C^{3,\alpha}( \Omega') }\le C_3(\gamma,  d(\Omega'))
 \end{align}
 and
 \begin{align}\label{c5}
 \| \psi_i^s\|_{C^{5,\alpha}(\Omega')} \le C_5(\gamma, \Omega'), ~\forall ~s\in [-\frac{1}{2}, 1].
 \end{align}

 From (\ref{c1-hi}), we have
 $$\|\dot \psi_i^s\|_{C^{1,\alpha}(\tilde\Omega_\gamma^i) }\le C_2(\gamma).$$
 By (\ref{c3-estimate}), we see that the coefficient of $\Delta_{\omega_{i}^s}$  is uniformly $C^{1,\alpha}$-bounded. Then applying the regularity for the uniformly parabolic equation (\ref{equation-flow}) in $ \tilde \Omega_\gamma^i\times [\frac{-1}{2},1]$, we get
 (\ref{c3-dot}).
 Thus  we can regarded  (\ref{s-equation-2}) as a complex Monge-Amp\`ere equation with  uniformly $C^{3,\alpha}$-bounded right term  $f$ such that
 $$\frac{(\tilde\omega_{i}+\sqrt{-1}\partial\bar\partial \psi^s_i)^n}{\tilde\omega_{i}^n}=e^{\dot \psi_i^s-\psi_i^s-\tilde h_i}=f, ~{\rm in}~ \tilde \Omega_\gamma^i.$$
  By  the regularity for the uniformly elliptic equation, we also obtain  (\ref{c5}) immediately.

 Repeating the above argument,  we will prove Proposition \ref{regularity-phi}.

 \end{proof}

 Since $ \Psi_i^*\tilde {\omega}_i$ is locally  uniformly $C^k$-convergent to  $\tilde\omega_\infty$,
   the Ricci potential  $\tilde h_i$ of $\tilde \omega_i$  converges to a smooth function $\tilde h_\infty$ on ${\rm reg}(\tilde M_\infty)$.
  On the other hand,  by  taking a diagonal subsequence of $ \psi_i^s(\Psi_\gamma)|_{\Omega_\gamma}$  in Proposition \ref{regularity-phi} for any $s\in  [-\frac{1}{2}, 1]$,  we get a smooth function $\psi_\infty$ on ${\rm reg}(\tilde M_\infty)$.
    Thus by  the equation (\ref{s-equation}),   the Ricci potential  $\dot \psi_i^s $ of $\omega_{t_i+s}$  is also  converges to
    a smooth limit  function $ h_\infty$ on ${\rm reg}(\tilde M_\infty)$.  Moreover,  $h_\infty$ is uniformly bounded in $\tilde M_\infty.$
   Hence, we get

 \begin{cor}\label{h-limit}    $\psi_\infty$ satisfies the following  complex Monge-Amp\`ere equation in  ${\rm reg}(\tilde M_\infty)$,
   \begin{align}\label{limit-equ}
   (\tilde\omega_{\infty}+\sqrt{-1}\partial\bar\partial \psi_\infty)^n= e^{\tilde h_\infty-\psi_\infty-h_\infty}  \tilde\omega_{\infty}^n.
 \end{align}

 \end{cor}

 Next section, we will show  that $(\omega_{\infty}+\sqrt{-1}\partial\bar\partial \psi_\infty)$ is a K\"ahler-Ricci soliton on ${\rm reg}(\tilde M_\infty)$ by using the monotonicity of Perelman's entropy $\lambda(\cdot)$ \cite{Pe}.

 \section{Proof of Theorem \ref{WZ} of Corollary  \ref{WZ-2}}

 Perelman's entropy $\lambda(\cdot)$ is based  on his  $\mathcal W$-functional \cite{Pe}. In our case, for $\omega_g\in 2\pi c_1(M, J)$,
 $\mathcal W$-functional can be expressed as (cf. \cite{TZZZ}),
 $$\mathcal W(g,f)=(2\pi)^{-n}\int_M (R+|\nabla f|^2+f)e^{-f} \omega_g^n,$$
 where $R(g)$ is a scalar curvature of $g$ and $(g,f)$  satisfies
  a normalization condition
 \begin{equation}
 \label{norm-1} \int_M e^{-f} \omega_g^n\,= \int_M (2\pi c_1(M))^n=V.
 \end{equation}
   Then
 $\lambda(g)$ is defined by
 \begin{equation*}
 \lambda(g)=\inf_{f}\{\mathcal W(g,f)| ~(g,f)~\text{satisfies}~ (\ref{norm-1})
 \}.
 \end{equation*}

 It is well-known that $\lambda(g)$ can be attained by some smooth function $f$ (cf. \cite{Ro}).
 In fact, such a $f$ satisfies the Euler-Lagrange equation of $\mathcal W(g,\cdot)$,
 \begin{equation}\label{EulerLagrange}
 \triangle f+f+\frac{1}{2}(R-|\nabla
 f|^{2})=(2\pi)^{n}V^{-1}\lambda(g).
 \end{equation}
 Following Perelman's computation in \cite{Pe}, we can deduce the first
 variation of $\lambda(g)$,
 \begin{equation}\label{variation}
 \delta\lambda(g)=-(2\pi)^{-m/2}\int_{M}<\delta g,
 \mathrm{Ric}(g)-g+\nabla^{2}f>e^{-f}\omega_{g}^n,
 \end{equation}
 where $\mathrm{Ric}(g)$ denotes the Ricci tensor of $g$ and
 $\nabla^{2} f$ is the Hessian of $f$. Hence, $g$ is a critical point
 of $\lambda(\cdot)$ in $2\pi c_1(M, J)$  if and only if $g$ is a gradient shrinking
 K\"ahler-Ricci soliton which satisfies
 \begin{equation}\label{RicciSolitons}
 \mathrm{Ric}(\omega_g)+\sqrt{-1}\partial\bar\partial f=\omega_g, ~\partial\partial f=0,
 \end{equation}
 where $f$ is  a  minimizer of $\mathcal W(g,\cdot)$.
 Namely, $\omega$ satisfies (\ref{kr-soliton}) with the holomorphic vector field
 $$v=g^{i\bar j}(-f)_{ \bar z^j }\frac{\partial }{\partial  z^i}.$$

 Denote the minimizer of evolved metric $g_t$  in (\ref{kr-flow}) by $f_t$.  Then by  (\ref{variation}), we get
 \begin{align}\label{derivative-lambda}\frac{d\lambda(g_t)}{dt}=\int_M ( |\partial\bar\partial (h_t+f_t)|^2
 + |\partial\partial f_t|^2)  \omega_{g_t}^n.
 \end{align}
 In particular,  $\lambda(g_t)=\mathcal W(g_t,f_t)$ is non-decreasing.

 The following lemma is due to \cite[Theorem 7.1]{TZZZ}.

 \begin{lem}\label{mini}
 There is a uniform constant $C$ such that
 \begin{align}\label{f-estimate}\|f_t\|_{C^0}+\|\nabla f_t\|_{\omega_t}+\|\triangle_{\omega_t}  f_t\|_{C^0}\le C, ~~\forall ~t>0.
 \end{align}
 \end{lem}

 We note that (\ref{EulerLagrange}) is equivalent  to
 \begin{equation}\label{EulerLagrange-2}
 \Delta v_t-\frac{1}{2}f_tv_t-\frac{1}{4}R(g_t)v_t=\frac{1}{2V}(2\pi)^{n}\lambda(g_t)v_t,
 \end{equation}
  where $v_t=e^{\frac{-f_t}{2}}$.
  Consider the  restricted equation of  (\ref{EulerLagrange-2}) on each $\tilde\Omega_\gamma^i$ for  metrics $\omega_{t_i+s}$ in Proposition \ref{regularity-phi}.
 Then by Lemma \ref{mini} and Proposition \ref{regularity-phi},   we get the higher order estimate  for $f_{t_i+s}$,
 \begin{align}\label{f-ck}\| f_{t_i+s}\|_{C^{k,\alpha}(\Omega')} \le A_\gamma( d(\Omega'), k), ~\forall ~s\in [\frac{1}{2}, 1].
 \end{align}

 \begin{lem}\label{limit-soliton}  There is a sequence of $s_i\in [\frac{1}{2}, 1]$ such that
  $((\Phi_i\cdot\Psi_\gamma^i)^{-1})^*\omega_{t_i+s_i}$ converges to a K\"ahler-Ricci soliton  $\omega_\infty$ on ${\rm reg}(\tilde M_\infty)$ in $C^\infty$-topology.
 \end{lem}

 \begin{proof}Since $\lambda(g_t)$ is uniformly bounded,  by (\ref{derivative-lambda}), it is easy to see that there is a sequence of $s_i\in [\frac{1}{2}, 1]$ such that
 $$\int_M ( |\partial\bar\partial (h_{t_i+s_i}-f_{t_i+s_i})|^2
 + |\partial\partial f_{t_i+s_i}|^2) dv\to 0,~ {\rm as}~i\to \infty.$$
 By the regularities in  Proposition \ref{regularity-phi} and (\ref{f-ck}), we may assume that  $f_{t_i+s_i}$ converges a smooth function $f_\infty$ on ${\rm reg}(\tilde M_\infty)$  while  $h_{t_i+s_i}$ converges to  $h_\infty$ as in
 Corollary \ref{h-limit}. Thus we get
 \begin{align}\label{h-f}f_\infty=-h_\infty+const., ~\partial\partial f_\infty=0, ~{\rm on}~ {\rm reg}(\tilde M_\infty).
 \end{align}
 The second relation implies that
 \begin{align}\label{x-vector} v=g_\infty^{i\bar j}( h_\infty)_{ \bar z^j }\frac{\partial }{\partial  z^i}
 \end{align}
 is a holomorphic vector field on  ${\rm reg}(\tilde M_\infty)$. Moreover,
 \begin{align}\label{kr-soliton-limit}\mathrm{Ric}(\omega_\infty)-\omega_\infty=\sqrt{-1}\partial\bar\partial h_\infty=L_v(\omega_\infty),
 \end{align}
 where $\omega_\infty=\tilde\omega_{\infty}+\sqrt{-1}\partial\bar\partial \psi_\infty$ and $\psi_\infty$ is the limit of
 $\psi_i^{s_i}$ as in Corollary  \ref{h-limit}. Hence,  $\omega_\infty$ is a K\"ahler-Ricci soliton on
 ${\rm reg}(\tilde M_\infty)$.

 \end{proof}

 For a   $L^\infty$-solution $\psi_\infty$ of (\ref{limit-equ})  as in Corollary \ref{h-limit}, we can regard it as a global   K\"ahler potential in $\tilde M_\infty$ by
 \begin{align}
 &\int_{\tilde M_\infty}    (\tilde\omega_{\infty}+\sqrt{-1}\partial\bar\partial \psi_\infty)^n
 = \int_{{\rm Reg} (\tilde M_\infty)}  (\tilde\omega_{\infty}+\sqrt{-1}\partial\bar\partial \psi_\infty)^n
 \notag\\
 &=\int_{\hat M_\infty}  (\pi^*(\tilde\omega_{\infty})+\sqrt{-1}\partial\bar\partial \psi_\infty(\pi))^n,
    \notag
    \end{align}
 where $\pi: \hat M_\infty\to \tilde M_\infty$  is a resolution  of   $\tilde M_\infty$.  By choosing a  sequence  of deceasing smooth functions $\psi_i\to \psi_\infty(\pi)$ on $\hat M_\infty$, we see that
 \begin{align}\label{full-mass}
  &\int_{\tilde M_\infty}    (\tilde\omega_{\infty}+\sqrt{-1}\partial\bar\partial \psi_\infty)^n\notag\\
 &= \lim_i\int_{\hat M_\infty}  (\pi^*(\tilde\omega_{\infty})+\sqrt{-1}\partial\bar\partial \psi_i)^n
   =\int_{\tilde M_\infty} \tilde\omega_\infty^n.
 \end{align}
 Thus     $\psi_\infty$  has full mass  \cite{BBEGZ}.

 \begin{prop}\label{q-fan}    The normal variety $\tilde M_\infty\subset \mathbb CP^N$ in Proposition  \ref{q-fano-limit} is  $Q$-Fano with $klt$- singularities
 and  the  $L^\infty$-solution $\psi_\infty$ of (\ref{limit-equ})  in Corollary \ref{h-limit}  has full mass with
  \begin{align}\label{volume-equiv}\int_{\tilde M_\infty} (\tilde\omega_\infty+\sqrt{-1} \partial\bar\partial \phi)^n
 =\int_{ M} \omega_0^n=V.
 \end{align}

 \end{prop}

 \begin{proof} By  (\ref{limit-equ}) and (\ref{kr-soliton-limit}), it is easy to see that
   \begin{align}\label{current-omega-1}\mathrm{Ric}(\tilde\omega_\infty)-\tilde\omega_\infty=\sqrt{-1}\partial\bar\partial \tilde h_\infty,
    ~{\rm in}~ {\rm reg}(\tilde M_\infty).
    \end{align}
 Moreover, by a result of Ding-Tian \cite{DT},  $\tilde h_\infty\in L^3(\tilde M_\infty)$.    Since $\psi_\infty$ is uniformly bounded on   ${\rm reg}(\tilde M_\infty)$,
  $$ \omega_\infty= \tilde\omega_{\infty}+\sqrt{-1}\partial\bar\partial \psi_\infty$$
  can be regarded as a  global current  and Monge-Amp\`ere
 measure  $ \omega_\infty^n$ is well-defined  on $\tilde M_\infty$.
 Thus (\ref{current-omega-1}) holds  for
 $\tilde \omega_\infty$ on
 $\tilde M_\infty$ as a current. This means that $K_{\tilde M_\infty}^{-l}$ is a restriction of the  line bundle
  $\mathcal {O}_{\mathbb CP^{N}}(1).$ Thus  $\tilde M_\infty$ is  a $Q$-Fano variety. In particular,   (\ref{volume-equiv})
  holds by (\ref{full-mass}).

 It remains to show that  $\tilde M_\infty$ has  $klt$- singularities.  Let $\tilde \theta_v$ and $\theta_v$ be the potentials of $v$ in (\ref{x-vector}) associated to metrics $\tilde\omega_\infty$ and  $\omega_\infty$, respectively. Then
  $$ \tilde \theta_v=  \theta_v-v(\psi_\infty) =h_\infty- v(\psi_\infty),~{\rm in}~ {\rm reg}(\tilde M_\infty),$$
  where $\psi_\infty$ is the solution of  (\ref{limit-equ}).
 On the other hand,  by (\ref{section-estiamte-1}),  we have
 $$\|\nabla\rho_l(M,\omega_i)\|_{\omega_i}\le C.$$
 It follows that
 $$|\nabla\psi_\infty|_{\omega_\infty}\le C.$$
 By the gradient estimate of $h_t$ in (\ref{h-t-estimate}), we also have
 $$ |\nabla h_\infty|_{\omega_\infty}\le C.$$
 Hence,
 \begin{align}\label{derivative} |v(\psi_\infty)|\le |\nabla\psi_\infty|_{\omega_\infty}^{\frac{1}{2}} |\nabla h_\infty|_{\omega_\infty}^{\frac{1}{2}}\le C,
 \end{align}
 and so,
  \begin{align}\label{x-potential}| \tilde \theta_v|\le C.
  \end{align}

 By (\ref{derivative}) and (\ref {x-potential}), we can use an argument in
   \cite[Proposition 3.8]{BBEGZ} to conclude that  a  $Q$-Fano variety, on which  there is a current solution of (\ref{current-omega-1} )  with full mass has klt singularities.
  In particular, $e^{\tilde h_\infty}\in L^p(\tilde M_\infty, \tilde\omega_\infty)$, for some $p>1$.

 \end{proof}

   Since
 $$e^{\tilde h_\infty-\psi_\infty-h_\infty}  \tilde\omega_{\infty}^n  $$
 is a Lebsegue measure,
 $  (\pi^*(\tilde\omega_{\infty})+\sqrt{-1}\partial\bar\partial \psi_\infty(\pi))^n$ can be extended to a global complex Monge-Amp\`ere
 measure  on  $\hat M_\infty$ such that
 \begin{align}
   (\pi^*(\tilde\omega_{\infty})+\sqrt{-1}\partial\bar\partial \psi_\infty(\pi))^n
   = e^{(\tilde h_\infty-\psi_\infty-h_\infty)(\pi)}  (\pi^*(\tilde\omega_{\infty}))^n= e^f \hat\omega^n,
   \notag
 \end{align}
 where $\hat\omega$ is a smooth K\"ahler metric and  $e^f$ is a $L^p$-function  ($p>1$)  on $\hat M_\infty$, respectively.
 By a result in \cite{EGZ}  (also see \cite[Lemma 3.6]{BBEGZ}),  we conclude  that $ \psi_\infty(\pi)$ is a continuous function on  $\hat M_\infty$.
 Thus  $\psi_\infty$  a continuous function on  $\tilde M_\infty$.   In subsection 4.2 below, we will give an alternative proof for the continuity of
   $\psi_\infty$ in case of $v=0$.

 The following shows that   the vector field $v$ in (\ref{x-vector})
 can be extended to a holomorphic  vector field in  $ \mathbb CP^N$.

 \begin{lem}\label{v-vaector}
    The vector field $v$ in (\ref{x-vector})
 can be extended to a holomorphic  vector field in  $ \mathbb CP^N$.

 \end{lem}

 \begin{proof}
 This  is in fact an application of Hartog's extension theorem as done in \cite{T2}.  For simplicity,   we let  $\mathcal S={\rm Sing}(\tilde M_\infty)$  and denote  $T_\epsilon(\mathcal S)$ to be an $\epsilon$-neighborhood of $\mathcal S$ in  $ \mathbb CP^N$.  Then there
 are finitely many open
 subsets $V_1,..., V_k$ of $\mathbb{C}P^N$ which covers   $\overline {T_\epsilon(\mathcal S)}$ so that  each $V_i$ is isometric to a ball in $\mathbb C^N$.

 Let  $\{s^\alpha_{t_{i}}\}$  be  a sequence of orthonormal  bases  of   $H^0(M,  K^{-l}_M, \omega_{t_i})$. Since  $|s^\alpha_{t_{i}}|_{h(\omega_{t_i})^{\otimes l}}$ is uniformly bounded,  by Proposition
  \ref{regularity-phi} and  (\ref{section-estiamte-1}),  it is easy to see that  each  $s^\alpha_{t_{i}}$ converges to a    holomorphic section  $s^\alpha_\infty$ of    $K^{-l}_{{\rm Reg}(\tilde M_\infty)}$. In fact, $\{s^\alpha_\infty\}$ becomes  an  orthonormal  basis   of   $H^0((\tilde M_\infty,  K^{-l}_{{\rm Reg}(\tilde M_\infty)})$ by (\ref{volume-convergence}) (to  see
  Remark \ref{relative-volume-convergence} below). Namely,
  $$(s^\alpha_\infty, s^\beta_\infty)=\int_{{\rm Reg}(\tilde M_\infty)} <s^\alpha_\infty, s^\beta_\infty>_{H_\infty}\omega_{\infty}^n=\delta_{\alpha}^\beta,$$
  where $H_\infty$ is the  induced Hermitian metric  by $\omega_\infty$ on $K_{{\rm Reg}(\tilde M_\infty)}^{-l}$.
  Moreover, by the partial $C^0$-estimate (\ref{partial-WZ}),  we have
  \begin{align}\label{partial-WZ-2}
 \rho_l(\tilde M_\infty,\omega_\infty)(x)=\Sigma_{i=0}^N \,|s_{\infty}^\alpha|^2_{H_\infty}(x)\ge A_0>0,~\forall~x\in {\rm Reg}(\tilde M_\infty).
 \end{align}
 Thus  for each $i$,  there is a section $\sigma_i$ in $H^0({\rm Reg}(\tilde M_\infty),  K^{-l}_{{\rm Reg}(\tilde M_\infty)})$ such that
 $$c\le |\sigma_i|_{H_\infty}\le c^{-1},~{\rm on}~\tilde M_\infty\cap V_i$$
  for some $c>0$ independent of $i$.

   Let    $Y$ be the imaginary part of $v$ .
 We integrate $Y$ to get a family of biholomorphic maps $\phi(t)$ from a neighborhood of $\overline{\tilde M_\infty\setminus T_\epsilon(\mathcal S)}$ into $\tilde M_\infty\setminus \mathcal S$, where $|t|\le \delta$ for some $\delta=\delta(\epsilon)>0.$
 Note that $\phi(0)=I$. Since $Y$ is a Killing field by  (\ref{kr-soliton-limit}),  whenever $\phi(t)$ is well-defined, it is an isometry of  $H_\infty$ on $K_{{\rm Reg}(\tilde M_\infty)}^{-l}$. Given any $\sigma\in H^0({\rm Reg}(\tilde M_\infty),  K^{-l}_{{\rm Reg}(\tilde M_\infty)}) $, $\phi_t^*(\sigma)$ is a bounded holomorphic section of $K_{{\rm Reg}(\tilde M_\infty)}^{-l}$ over $\tilde M_\infty\setminus T_\epsilon(\mathcal S)$. If $E$ is any subspace of $\mathbb CP^N$ of complex dimension $N-n+2$ with  (at most ) finite intersections with $\mathcal S$, then
 $M_E=\tilde M_\infty \cap E$  is a complex normal variety of complex dimension $2$ and $M_E\cap T_\epsilon(\mathcal S)$ is compact. For each $i$, $f_i=\frac{\phi_t^*(\sigma_i)}{\sigma_i}$ is a bounded holomorphic function on $(\tilde M_\infty\setminus T_\epsilon(\mathcal S))\cap V_i$, so by Hartog's extension theorem, $f_i$ extends to a  bounded holomorphic function on $M_E\cap V_i$. It follows that  $\phi_t^*(\sigma_i)$ extends to a  bounded holomorphic section of $ K_{{\rm Reg}(\tilde M_\infty)}^{-\ell}$ over $M_E\cap T_\epsilon(\mathcal S)$. Since $E$ is arbitrary,  we can easily deduce that $\phi_t^*(\sigma_i)$  extends to $\tilde M_\infty\cap T_\epsilon(\mathcal S)$.
 Thus $\phi_t$ lifts to an isomorphism of   $H^0({\rm Reg}(\tilde M_\infty),  K^{-l}_{{\rm Reg}(\tilde M_\infty)})$, or equivalently, $\phi_t$ is the restriction of an automorphism in $G=SL(N+1,\mathbb{C})$. Differentiating $\phi_t$ on $t$,  we see that $Y$, consequently, $v$ extends a holomorphic vector field on $\mathbb CP^N.$

 \end{proof}

 By Proposition \ref{q-fan} and Lemma \ref{v-vaector},  we introduce

 \begin{defi} Let  $(X, \omega_0)\subset \mathbb CP^N$ be   Q-Fano variety  with klt-singularities and  $v$ be  a holomorphic vector  field  on $ \mathbb CP^N$ which is tangent to $X$.  A current
 $$\omega=\omega_0+\sqrt{-1}\partial\bar\partial \varphi\in  [\omega_0]=2\pi c_1(X)$$
  is called a singular
  K\"ahler-Ricci  soliton  with respect to $v$ on   $X$ if   $\omega$  is smooth  on  ${\rm Reg}(X)$  with $L^\infty$  K\"ahler potential    $\varphi$ on $X$ such that
   it satisfies  K\"ahler-Ricci soliton equation on $X$ as a current,
 \begin{align}\label{singular-soliton}\mathrm{Ric}(\omega)-\omega= \sqrt{-1}\partial\bar\partial \theta_v=L_v\omega,
 \end{align}
 where $\theta_v$ is a bounded real potential of $v$ associated to $\omega$.

 \end{defi}

 Since $v(\varphi)$ is uniformly bounded in ${\rm Reg}(X)$,  one sees that  (\ref{singular-soliton}) is equivalent to a complex Monge-Amp\`ere equation,
 \begin{align}\label{soliton-2}
   (\omega_{0}+\sqrt{-1}\partial\bar\partial \varphi)^n
   = e^{h_0-\theta_0-v(\varphi)-\varphi }  \omega_{0}^n,
 \end{align}
 where  $\theta_0$ is a bounded potential of $v$ associated to $\omega_0$ and $h_0$ is a Ricci potential of $\omega_0$.
 By a   result of Berndtsson  \cite {Bo11} (also see \cite{BBEGZ}),    the uniqueness  theorem of  Tian-Zhu \cite{TZ1} for    K\"ahler-Ricci  solitons  can be generalized to   Q-Fano varieties  with klt-singularities.  Namely, if there are two solutions $\omega$ and $\hat \omega$  of (\ref{singular-soliton}), then there exists an element $\sigma\in {\rm Aut}_r(\tilde M_\infty)$ such that
 $$ \hat \omega=\sigma^*\omega,$$
 where  ${\rm Aut}_r(\tilde M_\infty)$ is the  reductive subgroup of ${\rm Aut}(\tilde M_\infty)$.

 \subsection{Singular structure of   $(X, d_\infty)$}  In this subsection, we study  the singular structure of   Gromov-Hausdroff  limit $(X, d_\infty)$ of   sequence of   K\"ahler metrics  $\eta_{t}$ on $M$  in  (\ref{modified-metric}).
 First we have

 \begin{lem}\label{eta-metric-local}Let $\{i\}$ be the sequence  in Proposition \ref{metric-equivalent}.  Then there is a uniform constant $C_\gamma'$ which depends
 only on  $\gamma$ such that
 \begin{align}\label{eta-metric}
 (C_\gamma')^{-1}\tilde\omega_{i}\le(\Phi_i^{-1})^*\eta_{t_i}\le C_\gamma' \tilde\omega_{i}, ~{\rm in}~\tilde \Omega_\gamma^i.
  \end{align}

 \end{lem}

 \begin{proof}By (\ref{yau-equation}) and (\ref{potenial-equation}), we have
 \begin{align}\label{yau-equation-1}
 \eta_{t_i}^n= (\tilde \omega_i + \sqrt{-1} \partial\bar\partial (\kappa_{t_i}+\psi_i^0))^n=e^{\tilde h_i-\psi_i^0} \tilde\omega_{i}^n,  ~{\rm in}~\tilde \Omega_\gamma^i.
 \end{align}
 By the gradient estimate of $\bar s_{t_i}^\alpha$ and the lower bound of $\rho_l(M,\eta_{t_i})$  together with  (\ref{matrix}),  we know that
 \begin{align} \label{gradient-eta-metric}(\Phi_i^{-1})^*\eta_{t_i}\leq C \tilde \omega_i.
 \end{align}
 Thus by  (\ref{yau-equation-1})  and  the relation  (\ref{phi-s-c0}),  we also get
 $$C_\gamma^{-1}\tilde\omega_{i}\le(\Phi_i^{-1})^*\eta_{t_i},~{\rm in}~\tilde \Omega_\gamma^i.$$

 \end{proof}

 By  Proposition \ref{regularity-phi},  there is a  subsequence of $\psi_i^0$   (still denoted  by the same)  such that
   \begin{align}\label{c-infty-convrergence-eta} \psi_i^0((\Phi_i\Psi_\gamma^i)^{-1}) \to \psi_\infty^0, ~  {\rm on}~  {\rm Reg}(\tilde M_\infty),
   \end{align}
 in  local $C^\infty$-topology   as $i,\gamma\to\infty$, where  $\psi_\infty^0$ is  smooth  uniformly bounded on
   ${\rm Reg}(\tilde M_\infty)$.
 On the other hand,  by Lemma  \ref{eta-metric-local}, we can  apply the regularity for  uniformly elliptic equations  to (\ref{yau-equation-1}) to see  that
 $$\| (\kappa_{t_i} +\psi_i^0)\|_{C^{k,\alpha} (\Omega')}\le C_\gamma( d(\Omega'), k), ~\forall ~k, \Omega'\subset\tilde {\Omega}_\gamma^i. $$
   Thus    there is a  subsequence of $\kappa_{t_i}$   (still denoted  by the same)  such that
   \begin{align}\label{kappa-conv} \kappa_{t_i} \to  \kappa_\infty ~  {\rm on}~  {\rm Reg}(\tilde M_\infty)
   \end{align}
 in  local $C^\infty$-topology   as $i\to\infty$,  and  $( \kappa_\infty+\psi_\infty^0)$
    satisfies  complex Monge-Amp\`ere equation,
   \begin{align}\label{yau-equation-eta}
 \eta_{\infty}^n=(\tilde\omega_\infty+ \sqrt{-1}\partial\bar\partial(\kappa_\infty +\psi_\infty^0))^n= e^{\tilde h_\infty-\psi_\infty^0}  \tilde\omega_{\infty}^n.~{\rm Reg}(\tilde M_\infty) .
 \end{align}
 Moreover,  by (\ref{c0-psi-kappa}),  $(\kappa_\infty +\psi_\infty^0)$ can be extended to a $L^\infty$-solution on  $\tilde M_\infty$.

 The following proposition  improves the regularity of  limit $(X, d_\infty)$ in Theorem \ref{normal} for the sequence $\{\eta_i\}$.

 \begin{prop}\label{corollary-eta} Let $\eta_{t_i}$  be a  subsequence  of K\"ahler metrics on $M$  as in (\ref{eta-metric}) and  $(X, d_\infty)$  the Gromov-Hausdroff  limit of $\eta_{t_i}$ as in  Theorem \ref{normal}. Then  the completion   $\overline{({\rm Reg}(\tilde M_\infty), \eta_\infty)}$   of $({\rm Reg}(\tilde M_\infty), \eta_\infty)$ is isometric to  $(X, d_\infty)$.
 Moreover, $X$ is homeomorphic to $\tilde M_\infty $ and
 \begin{align}\label{singular-cod}{\rm codim}( {\rm Sing}(X, d_\infty))\geq4.
 \end{align}

 \end{prop}

 \begin{proof} Since
  $(\kappa_\infty+\psi_\infty^0)$ is a  $L^\infty$ K\"ahler potential  on $\tilde M_\infty$ as $\psi_\infty$  in Proposition \ref{q-fan},  as in (\ref{volume-equiv}), we have
  $$\int_{\tilde M_\infty} \eta_\infty^n=\int_{M} \omega_0^n=V.$$
  Namely, $\psi_\infty^0$ has  full mass.   Thus     by (\ref{c-infty-convrergence-eta}) and (\ref{kappa-conv}),  we obtain
 \begin{align}\label{volume-convergence-eta}\int_{(\Phi_i^{-1} \Psi_\gamma^i)(\Omega_{\gamma})} \eta_{t_i}^n \to V,~{\rm as}~i,\gamma\to\infty.
 \end{align}

 Note that  $({\rm Reg}(\tilde M_\infty),  \eta_\infty)$ is  an open set of  smooth part of the Gromov-Hausdroff  limit    $(X,  d_\infty)$  of $(M, \eta_{t_i})$.   Then by taking  any Cauchy  sequence in ${{\rm Reg}(\tilde M_\infty, \eta_\infty)}$ one can complete it so that its  completion $\overline {{\rm Reg}(\tilde M_\infty, \eta_\infty)}$ is a subset of  $(X,  d_\infty)$.  Suppose that
 \begin{align}\label{two-space-neq-1}\overline {{\rm Reg}(\tilde M_\infty, \eta_\infty)}\subset\subset (X,   d_\infty),
  \end{align}
  which means that  there is an open set $U \subset (X,   d_\infty)$ such that
  $$\overline{ U}\cap \overline {{\rm Reg} (\tilde M_\infty, \eta_\infty)}=\emptyset.$$
   It follows that  there is a sequence of $r$-geodesic balls $B_r(\omega_{t_i})\subset (M, \omega_{t_i})$ which converges to an open set  $V\subset\overline{ U}$ in  Gromov-Hausdroff topology.
  Thus for any fixed $\gamma$, there is an $i_\gamma$ such that
  $$(\Phi_i^{-1} \Psi_\gamma^i)(\Omega_{\gamma}) \subset \tilde M_i\setminus B_r(\omega_{t_i}),~\forall ~i\ge i_\gamma.$$
  Hence, by  (\ref{volume-convergence-eta}),  we derive
  \begin{align}\label{full-volume}\lim_i  \int_{\tilde M_i\setminus B_r(\omega_{t_i})}  \eta_{t_i}^n=V.
  \end{align}
  On the other hand, by Zhang's result \cite{ZhK},
  $${\rm vol}_{\eta_{t_i}}( B_r(\eta_{t_i}))\ge cr^{2n}, $$
 which implies that
  \begin{align}
 {\rm vol}_{\eta_{t_i}}(\tilde M_i\setminus B_r(\eta_{t_i}))\le V-cr^{2n}.\notag
 \end{align}
 But  this  is  a contradiction with  (\ref{full-volume})!    Therefore,   it must be
 \begin{align}\label{two-space-1}\overline {{\rm Reg}(\tilde M_\infty, \eta_\infty)}=(X,  d_\infty).
  \end{align}

 $X$ is homeomorphic to $\tilde M_\infty $ by Theorem \ref{normal} (also see Remark \ref{LS-theorem}).  In fact, as in the of Proposition \ref{q-fano-limit},
 the Kodaira embedding   $\bar \Phi_i$  of $M$ given   by the orthonormal basis $\{\bar s_{t_i}^\alpha\}$ of $H^0(M, K_M^{-l}, \eta_{t_i})$
  are uniformly Lipschitz by the gradient estimate of  $\bar s_{t_i}^\alpha $ as in (\ref{section-estiamte-1}).   Thus  we have a limit map $\bar\Phi_\infty$  which gives   a homeomorphism from $(X,d_\infty)$ to $(\tilde M_\infty,\tilde \omega_\infty)$.

 We are left to prove (\ref{singular-cod}).   Denote the regular part of $(X,d_\infty)$ by $\mathcal R$ which consists of  points with flat  tangent cones   \cite{CC}.  We claim  that
  \begin{align}\label{regular-part-R}\bar\Phi_\infty(\mathcal R)={\rm Reg}(\tilde M_\infty).
  \end{align}

  By Proposition 2.4 in \cite{LS}, for any point $x\in \mathcal R$, there is a neighborhood $U_x$ around $x$ such that
  the image $\bar\Phi_\infty(U_x)\subset {\rm Reg}(\tilde M_\infty)$.  In particular,  we get $\bar \Phi_\infty(x)\in {\rm Reg}(\tilde M_\infty)$. On the other hand, for any $p\in {\rm Reg}(\tilde M_\infty)$, choose a small convex ball $B_r(p,\eta_\infty)$ inside the open Riemannian manifold $({\rm Reg}(\tilde M_\infty),\eta_\infty)$. Then for $x=\bar\Phi_\infty^{-1}(p)$,    $B_r(p,\eta_\infty)$ is isometric to
   $B_r(x, d_\infty)$ as a length space. As a consequence, we have $x\in \mathcal R$. Thus
  (\ref{regular-part-R}) is true.

  By  (\ref{regular-part-R}) and the homomorphism $\bar \Phi_\infty$,  we have
  $${\rm Sing}(X,  d_\infty)=\bar\Phi_\infty^{-1}({\rm Sing}(\tilde M_\infty))\subset \mathcal S_{2n-2},$$
   where $\mathcal S_k$ is the stratification of ${\rm Sing}(X, d_\infty)$ in \cite{CC}.
 It remains to show that
 \begin{align}\label{2n-2-singular}\mathcal S_{2n-2}\setminus \mathcal S_{2n-4}=\emptyset.
 \end{align}
 Note that $\mathcal S_{2n-3}\setminus \mathcal S_{2n-4}=\emptyset$ \cite{CCT}.   By definition,     for any $x\in \mathcal S_{2n-2}\setminus \mathcal S_{2n-4}$,  there are  two sequences $\{\epsilon_i\}$ and $\{r_i \}$ to $0$ such that $r$-distance ball  $B_x(r_i)$ around $x$ satisfies
 $${\rm dist}_{GH}(B_x(r_i), B_V(o, r_i))\le r_i\epsilon_i,$$
 where $B_V(o, r)$ is a $r$-ball of metric cone $(V,o)$ splitting off $\mathbb R^{2n-2}$.  On the other hand, by Proposition 3.2 in \cite{LS}, we know that
 there exists a small ball  $B_x(r_i)$ such that  $\bar\Phi_\infty(B_x(r_i) )\subset {\rm Reg}(\tilde M_\infty)$.   In particular,   $\bar \Phi_\infty(x)\in {\rm Reg}(\tilde M_\infty)$. Thus $x\in \mathcal R$ by  (\ref{regular-part-R}) and the homomorphism $\bar\Phi_\infty$, and so   (\ref{2n-2-singular}) must be true.  As a sequence,
 ${\rm Sing}(X, d_\infty)=\mathcal S_{2n-4}$. Hence,   by Theorem 4.7 in \cite{CC}, we prove
 $${\rm codim}( {\rm Sing}(X, d_\infty))\geq4.$$

 \end{proof}

 \subsection{Proof of Theorem \ref{WZ}}  We are ready  to prove Theorem \ref{WZ}. First we prove the convergence of $\omega_{t_i}$.

 \begin{lem}\label{same-KE}Let    $\omega_\infty$  be   the   singular  K\"ahler-Ricci  soliton on $\tilde M_\infty$
  in  Lemma  \ref{limit-soliton}.     Then  $\omega_{t_i}$ (perhaps after  taking a subsequence)  converges to  $\sigma_0^*\omega_\infty$  on $\tilde M_\infty$ for some  $\sigma_0\in {\rm Aut}_r(\tilde M_\infty)$.
 \end{lem}

 \begin{proof}As in the proof of  Lemma \ref{limit-soliton},  for any $\epsilon_0>0$,
  there is a sequence of $s_i\in [-\epsilon_0, 0]$ such that
  $((\Phi_i\cdot\Psi_\gamma^i)^{-1})^*\omega_{t_i+s_i}$ converges to a  K\"ahler-Ricci soliton  $\bar\omega_\infty$ on  ${\rm Reg}(\tilde M_\infty)$ in $C^\infty$-topology.   Moreover, according to the proof of Proposition \ref{q-fan}, $\bar\omega_\infty$ is a   singular K\"ahler-Ricci soliton   on  $\tilde M_\infty$ with
  \begin{align}\label{volume-bar}
  \int_{\tilde M_\infty}\bar\omega_\infty^n=V.
  \end{align}
 Thus by   the uniqueness  result of Berndtsson  for  singular  K\"ahler-Ricci solitons  \cite {Bo11},
 $ \bar\omega_\infty=\sigma^*\omega_\infty$ for some  $\sigma\in {\rm Aut}_r(\tilde M_\infty)$.
 Since $ ((\Phi_i \Psi_\gamma^i)^{-1})^*\omega_{t_i+s_i}$ is locally $C^k$-uniformly bounded,  we  see that there exists a uniform constant $C_0$ such that
 \begin{align}\label{simga-bounded}{\rm dist}( \sigma, {\rm Id})\le C_0.
 \end{align}

 On the other hand,  by   (\ref{equation-flow}) together with (\ref{c3-dot}), we have
 $$\|\dot \psi_i^s\|_{C^{k,\alpha}(\tilde\Omega_\gamma^i) }\le C_k(\gamma),~\forall ~s\in [-\frac{1}{2}, 1].$$
 Thus the  K\"ahler potential of $ ((\Phi_i \Psi_\gamma^i)^{-1})^*\omega_{t_i+s_i}$  and its derivatives with any orders  are  locally uniformly continuous.  Hence, by taking  a diagonal subsequence of  $\omega_{t_i+s_i}$ with $s_i\to 0$, we get
 from (\ref{simga-bounded}),
 \begin{align} \lim_{i_k} ((\Phi_{i_k} \Psi_\gamma^{i_k})^{-1})^*\omega_{t_{i_k}}= \lim_{i_k} ((\Phi_{i_k} \Psi_\gamma^{i_k})^{-1})^*\omega_{t_{i_k}+s_{i_k}},  ~\forall~
 \Omega_\gamma\subset \tilde M_\infty, \notag
 \end{align}
 and
 $((\Phi_{i_k} \Psi_\gamma^{i_k})^{-1})^*\omega_{t_{i_k}+s_{i_k}}$
  converges to  $\sigma_0^*\omega_\infty$  for  some  $\sigma_0\in {\rm Aut}_r(\tilde M_\infty)$  on  ${\rm Reg}(\tilde M_\infty)$ in $C^\infty$-topology.

 \end{proof}

   \begin{proof}[Proof of Theorem \ref{WZ}]
   The local  $C^\infty$-convergence of   $\omega_{t_i}$ to   $\omega_\infty$ on  ${\rm Reg}(\tilde M_\infty)$  follows from
   Lemma  \ref{same-KE} with   $\omega_\infty$ replaced by $\sigma_0^*\omega_\infty$.
      In fact, there is a subsequence of $i$  ( still denoted by $i$ for simplicity) such that
   \begin{align}\label{c-infty-convrergence}((\Phi_{i} \Psi_\gamma^{i})^{-1})^*\omega_{t_{i}}\to \omega_\infty,  {\rm on}~  {\rm Reg}(\tilde M_\infty),
   \end{align}
   as $i,\gamma\to\infty$.
   Since $({\rm Reg}(\tilde M_\infty),  \omega_\infty)$ is  an open set of  smooth part of the Gromov-Hausdroff  limit    $(M_\infty,  \omega_\infty')$  of $(M, \omega_{t_i})$,    the completion $\overline {{\rm Reg}(\tilde M_\infty)}$ of  $({\rm Reg}(\tilde M_\infty),  \omega_\infty)$ is contained in  $(M_\infty,  \omega_\infty')$.
 On the other hand,  by (\ref{volume-equiv}),  we have
  $$\int_{ \tilde M_\infty} \omega_\infty^n=\int_{\tilde M_\infty} \tilde\omega_\infty^n=V.
 $$
 Thus   by (\ref{c-infty-convrergence}), we get
 \begin{align}\label{volume-convrrgence-m}\int_{(\Phi_i^{-1} \Psi_\gamma^i)(\Omega_{\gamma})} \omega_{t_i}^n \to V,~{\rm as}~i,\gamma\to\infty.
 \end{align}

 Suppose that
 \begin{align}\label{two-space-neq}\overline { ({\rm Reg}(\tilde M_\infty), \omega_\infty)}\subset\subset (M_\infty,  \omega_\infty').
  \end{align}
  Then there is an open set $U \subset (M_\infty,  \omega_\infty')$ such that
  $$\overline{ U}\cap \overline { ({\rm Reg}(\tilde M_\infty), \omega_\infty)}=\emptyset.$$
  Thus as in the proof of  Corollary \ref{corollary-eta},
   there is a sequence of $r$-geodesic balls $B_r(\omega_{t_i})\subset (M, \omega_{t_i})$  such that
     $$\lim_i  \int_{\tilde M_i\setminus B_r(\omega_{t_i})}  \omega_{t_i}^n=V.$$
  However, by the Perelaman's  non-collapsed property in   Lemma \ref{lem:perelman-1}-1),  we have
 \begin{align}\label{less-volume}
 {\rm vol}_{\omega_{t_i}}(\tilde M_i\setminus B_r(\omega_{t_i}))\le V-cr^{2n}.
 \end{align}
 Therefore, we obtain a contradiction from the above two relations and prove that
 \begin{align}\label{two-space}\overline {({\rm Reg}(\tilde M_\infty, \omega_\infty))}=(M_\infty,  \omega_\infty').
  \end{align}

  In case that $\omega_\infty$ is a singular  K\"ahler-Einstein metric on   $\tilde M_\infty$,  by the relation  (\ref{yau-equation-eta}), we get
  $${\rm Ric}(\eta_\infty)=\omega_\infty={\rm Ric}(\omega_\infty),  ~{\rm on}~{\rm Reg}(\tilde M_\infty).$$
 Moreover, both of $\eta_\infty$ and $\omega_\infty$ are  satisfied  globally as a current  in  $[\tilde \omega_\infty]$ on $\tilde M_\infty$
 by  the following complex Monge-Amp\`ere equation,
  $$\eta_\infty^n=\omega_\infty^n =e^f\tilde\omega_\infty^n,$$
   where $e^f$ is a $L^p$-function  ($p>1$)  with respect to $\tilde\omega_\infty$.
  Thus by the uniqueness  of solutions of complex Monge-Amp\`ere equation \cite{BT, EGZ},
  $$\omega_\infty=\eta_\infty,  ~{\rm on}~\tilde M_\infty.$$
  Hence,  by  (\ref{two-space-1}) and  (\ref{two-space}), we prove
 \begin{align}\label{3-topology}(M_\infty, \omega_\infty')=\overline {({\rm Reg}(\tilde M_\infty), \omega_\infty)}=(X,d_\infty).
 \end{align}
 By  Proposition \ref{corollary-eta},  we also get
 \begin{align}\label{singular-cod-m-infty}{\rm codim}( {\rm Sing}(M_\infty, \omega_\infty'))\geq4.
 \end{align}

 By (\ref{c0-psi})  and (\ref{section-estiamte-2}),   it is easy to see that the potential $\psi_\infty$ of $\omega_\infty$ can be extended to a
  continuous function on $(M_\infty, \omega_\infty')$.   Thus by    Proposition \ref{corollary-eta} together with (\ref{3-topology}),   the extended $\psi_\infty$  is also  a  continuous function on $\tilde M_\infty$.  The proof is complete.

   \end{proof}

 \begin{rem}By Theorem \ref{WZ},   the Gromov-Hausdorff  limit $ (M_\infty, \omega_\infty')$ of $\omega_{t_i}$  is a  K\"ahler-Ricci soliton on the open set  $ {\rm Reg}(\tilde M_\infty)$. Moreover, if   $\omega_\infty'$ is a  K\"ahler-Einstein metric  on
  $ {\rm Reg}(\tilde M_\infty)$,
 $${\rm Sing}(M_\infty, \omega_\infty')={\rm Sing}(\tilde M_\infty, \tilde\omega_\infty)
 ~{\rm and}~\rm {codim}( {\rm Sing}(M_\infty, \omega_\infty'))\geq4.$$
 It is still unknown whether the Hausdroff measure of  codimension $4$  of ${\rm Sing}(M_\infty, \omega_\infty')$ is finite or not.
 \end{rem}

 \begin{rem}\label{relative-volume-convergence} By (\ref{two-space}),  we  introduce  the relative volume of any set $U$ in $(M_\infty,  \omega_\infty')$ by
  $${\rm vol}(U)=\int_{{\rm Reg}(\tilde M_\infty) \cap U} (\omega_\infty')^n.$$
  Then by  the local convergence of  ${\omega_{t_i}}$,  it is easy to see that for any
  $$(U_i, \omega_{t_i})\to (U,\omega_\infty')$$
  in Gromov-Hausdroff topology, it holds
  \begin{align}\label{volume-convergence}
  {\rm vol}_{\omega_{t_i}}(U_i) \to {\rm vol}(U).
  \end{align}
 \end{rem}

  The relation (\ref{volume-convergence}) will be used in  Subsection 4.3.

   \subsection{Proof of  Corollary \ref{WZ-2}}

 Recall  the Mabuchi $K$-energy,
  $$\mu(\varphi)
  =-\frac {n}{ V}\int_0^1\int_M
 \dot\psi ({\rm R}(\omega_{\psi})-n)
 \omega_{\psi}^{n} \wedge dt,
 $$
  where $\psi=\psi_t$ $(0\le t\le 1)$ is a  path  of K\"ahler potentials in $2\pi c_1(M)$  connecting $0$ to $\varphi$.
  In case that $\mu(\cdot)$  is bounded below,  it was proved that $f_t$ in (\ref{f-estimate}) satisfies   \cite[Proposition 4.2]{TZZZ},
   \begin{align}\label{f-0}\|f_t\|_{C^0(M)}\to 0, ~{\rm as}~t\to\infty.
   \end{align}

   \begin{defi}[\cite{Ti97}]    Let     $M$ be a  Fano manifold.  Denote  $F_{M}(v)$ to  be Ding-Tian's generalized Futaki-invariant
  for $C^*$-action $G_0$ induced by $v\in sl(N+1, \Bbb C)$ with  a $Q$-Fano variety   as its center fiber.
  $M$ is called

 i)  $K$-semistability if $F_{M}(v)\ge 0$ for any $v$;

 ii)  $K$-stability if $F_{M}(v)> 0$ for any $v$;

 iii)  $K$-polystability if $F_{M}(v)\ge 0$  for  any $v$  and $"="$  holds if and only if
 $v$ lies  in  the Lie algebra of  ${\rm Aut}(M).$

 \end{defi}

 \begin{proof}[Proof of  Corollary \ref{WZ-2}]By the stability theorem \cite{Zhu}, it suffices to prove that there are  sequences of
  $\{\omega_{t_i}\}$  in  (\ref{kr-flow})  and  automorphisms $\Psi_i$ in ${\rm Aut}(M)$  such that $\Psi_i^*\omega_{t_i}$  converges smoothly  to a
   K\"ahler-Einstein  metric on  $M$.  However,  By Theorem \ref{WZ},   any sequences $\{\omega_{t_i}\}$ (perhaps after taking a subsequence)  converges  to a singular  K\"ahler-Ricci soliton  $\omega_\infty $ on  a $Q$-Fano variety  $\tilde M_\infty$ with klt-singularities.
   By Lemma \ref{limit-soliton},
      together with (\ref{h-f}) and (\ref{f-0}),
  we see that  $\omega_\infty$   is in fact a singular  K\"ahler-Einstein metric on  $\tilde M_\infty$.  Thus by Tian's generalized Matsushima theorem \cite{T2},
  $\tilde M_\infty$ must be reductive.
   As a consequence, by  a well-known GIT  result  in algebraic geometry, there exists a $C^*$-action  $G_0=\{\sigma_t(v)\}\subset SL(N+1, \Bbb C)$   such that $\sigma_t(v)$ degenerates to $\tilde M_\infty$ as $t\rightarrow \infty$.
 Since $\tilde M_\infty$ is a  K\"ahler-Einstein variety,   Ding-Tian's generalized Futaki-invariant $F_{M}(v)$ will vanish for the action  $G_0$ \cite{DT}.

  We claim that  $\tilde M_\infty$ is  biholomorphic to $M$.  Then by the regularity of singular  K\"ahler-Einstein metrics  (cf. \cite{JWZ, BBEGZ}),   $\omega_\infty$ is in fact  a  smooth K\"ahler-Einstein  manifold  on $M$ and  the corollary follows.   On contrary,   the  $C^*$-action  $G_0$ will be  nontrivial, i.e, $G_0$
  is not a one-parameter  subgroup in ${\rm Aut}(M)$. Thus
   by  the condition of $K$-stability,  we have
  $$F_M(v)> 0.$$
  But this is impossible!    Corollary \ref{WZ-2}  is proved.

  \end{proof}

  \subsection{Computation of  $L(g)$} In this subsection, we give another application of Theorem \ref{WZ} to the computation of
   energy level $L(g)$ of  Perelman's entropy along a   K\"{a}hler-Ricci flow.

   \begin{defi} [\cite{TZ4}]  Let  $(M,\omega_t)$ be a    K\"{a}hler-Ricci flow in  (\ref{kr-flow}) with initial metric  $\omega_0=\omega_g$.  The energy level $L(g)$ of entropy
 $\lambda(\cdot)$ for    $(M,\omega_t)$   is
 defined by
 \begin{equation*}
 L(g)=\lim_{t\rightarrow\infty}\lambda(g_t).
 \end{equation*}
 \end{defi}

 By the monotonicity of $\lambda(g_t)$, we see that $L(g)$ exists and
 it is finite.   We shall  give an explicit  computation of  $L(g)$.

  By (\ref{kr-soliton-limit}),  $h_\infty$ is a potential of holomorphic vector field $v$.  As in (\ref{singular-soliton}),  we write it by $\theta_v$. Then
  by the estimate (\ref{f-estimate}) in Lemma \ref{mini},   we have
 \begin{align}\label{c2-theta}\|\theta_v\|_{C^0}+\|\nabla \theta_v\|_{\omega_\infty}+\|\triangle_{\omega_\infty}  \theta_v\|_{C^0}\le C.
 \end{align}
 We will normalize it by
 $$\int_{\tilde M_\infty}  e^{\theta_v}\omega_\infty^{n}=\int_M \omega_0^n=V.
 $$

  We recall a notation  in \cite{TZZZ} by
  $$
 N_v(\omega_\infty)=\int_{\tilde M_\infty}\theta_ve^{\theta_v}\omega_{\infty}^{n}.
 $$
  By  Jensen's inequality, it is easy to see that
 $$N_v(\omega_\infty)\ge 0$$
 and $"="$ holds
  if and only if  $v=0$.

  \begin{lem}\label{entropy-computation} Let $ (\omega_\infty, v)$ be  a singular K\"ahler-Ricci soliton limit  of    (\ref{kr-flow}) with   an initial metric  $\omega_g$  in  Lemma \ref{mini}. Then
  \begin{align}\label{lambda-quantity}
  L(g)= (2\pi)^{-n}[nV-N_v(\omega_\infty)].
  \end{align}

  \end{lem}

  \begin{proof}
  We note by  (\ref{c2-theta}) that
  $$\mathcal W(\omega_\infty,  -\theta_v)=(2\pi)^{-n}\int_{\tilde M_\infty}  (  R(\omega_\infty)+  |\nabla \theta_v|^2-\theta_v)e^{\theta_v} \omega_\infty^n$$
 is well-defined,
 where $R(\omega_\infty)= n+ \triangle_{\omega_\infty}  \theta_v$.
 By  stoke's formula,
 $$\int_{\tilde M_\infty}(\Delta
 \theta_{v}+|\nabla
 \theta_{v}|^{2})e^{\theta_{v}}\omega_{\infty}^{n}=0.$$
 Thus
 \begin{eqnarray*}
  \mathcal W(\omega_\infty,  -\theta_v)=(2\pi)^{-n}[nV-N_v(g)].
  \end{eqnarray*}
  On the other hand, by (\ref{f-ck}) together with the relation (\ref{volume-convergence}), it is easy to see that
  $$ \lim_{t_i}\lambda(g_{t_i})=\mathcal W(\omega_\infty,  -\theta_v).$$
  Hence  by the monotonicity   of $\lambda(g_t)$,     we get (\ref{lambda-quantity}).

  \end{proof}

  The following proposition give a weak rigidity  for the limit of  K\"ahler-Ricci flow   (\ref{kr-flow}).

  \begin{prop}\label{two-equiv}Let   $(M,\omega_t)$ be  a K\"ahler-Ricci flow   (\ref{kr-flow}) with   an initial metric  $\omega_g$.  Then the  following  two statements  are true.

 1)  There is a  sequence of  $(M,\omega_t)$   such that the limit  is a  singular  K\"ahler-Einstein metric (or K\"ahler-Ricci soliton)   in the Hamilton-Tian conjecture  if and only if any  limit   in  $(M,\omega_t)$  is a  singular  K\"ahler-Einstein metric (or K\"ahler-Ricci soliton).

  2)  If   there is a  sequence of  $(M,\omega_t)$     such that the limit  is a  singular  K\"ahler-Einstein metric    in the Hamilton-Tian conjecture, then
  \begin{align}\label{KE-entropy}\sup\{\lambda(g')|~\omega_{g'}\in 2\pi c_1(M,J)\}=(2\pi)^{-n}V.
 \end{align}

  \end{prop}

  \begin{proof} 1).   By Theorem \ref{WZ},   any sequence $\{\omega_{t_i}\}$ (perhaps after taking a subsequence)    in  $(M,\omega_t)$ converges  to a singular  K\"ahler-Ricci soliton  $(\omega_\infty, v) $ on  a $Q$-Fano variety  $\tilde M_\infty$ with klt-singularities.  By  (\ref{lambda-quantity}) in   Lemma \ref{entropy-computation},    $L(g)$ is independent of choice of sequences.   In other words,
   $ N_v(\omega_\infty)$ is  independent of  limit $\omega_\infty$.   Moreover,   $ N_v(\omega_\infty)=0$  if and only if  $\omega_\infty$ is a   singular  K\"ahler-Einstein metric. Hence, 1) is true.

  2).  In case that there is a   limit  of   singular  K\"ahler-Einstein  metric   ($M_\infty, \omega_\infty)$ for some  sequence  of  $(M,\omega_t)$,   by Lemma \ref{entropy-computation},
  $$ \lim_t\lambda(g_t)=(2\pi)^{-n}V.$$
  Note that
  $$\lambda(g')\le  \mathcal W(\omega_{g'},  0)=(2\pi)^{-n}V, ~\forall ~\omega_{g'}\in 2\pi c_1(M, J).$$
  Thus we get (\ref{KE-entropy}) immediately.

    Based on Proposition  \ref{two-equiv},  we propose the following conjecture.

    \begin{conj}\label{semi-stable-entropy} A Fano  manifold is  $K$-semistable if and only if    (\ref{KE-entropy}) holds.

    \end{conj}

  Conjecture  \ref{semi-stable-entropy} gives an analytic character for  a $K$-semistable Fano manifold in terms of Perelaman's entropy.
  We note that the necessary part is true according to the proof in Corollary \ref{WZ-2} since  the $K$ semi-stability
  is equivalent to the lower bound of $K$-energy by a theorem of Li-Sun \cite{LS}.

  \end{proof}

 \end{document}